\documentclass{amsart}
\usepackage{graphicx}
\usepackage{cite}
\newtheorem{thm}{Theorem}[section]
\newtheorem{cor}[thm]{Corollary}
\newtheorem{lem}[thm]{Lemma}
\newtheorem{prop}[thm]{Proposition}
\theoremstyle{definition}

\theoremstyle{remark}

\numberwithin{equation}{section}

\begin{document}

\title[A priori estimates and Blow-up behavior]{A priori estimates and Blow-up behavior
for solutions of $-Q_{N}u=Ve^{u}$ in bounded domain in $\mathbb{R}^{N}$ }%
\author{Rulong Xie and Huajun Gong}%
\address{Rulong Xie. School of Mathematical Sciences, University of
Science and Technology of China, Hefei 230026, China}
\email{xierl@mail.ustc.edu.cn;}%
\address{ Huajun Gong (Corresponding author). School of Mathematical Sciences, University of
Science and Technology of China, Hefei 230026, China}%
\email{huajun84@hotmail.com.}

\thanks{The authors would like to
thank Prof. Jiayu Li and Dr. Chao Xia for helpful discussion and suggestions.
This work is supported by the Excellent Young Talent Foundation of Anhui Province (2013SQRL080ZD)}%
\subjclass{35J60, 53C60}%
\keywords{Blow-up behavior, $N$-anisotropic Laplacian, $N$-anisotropic Liouville equation. }%

\begin{abstract}
 Let $Q_{N}$ be $N$-anisotropic Laplacian operator, which contains the ordinary  Laplacian operator, $N$-Laplacian operator and anisotropic Laplacian operator. In this paper, we firstly obtain the properties for $Q_{N}$, which contain the weak maximal principle, the comparison principle and the mean value property. Then a priori estimates and blow-up analysis for solutions of $-Q_{N}u=Ve^{u}$ in bounded domain in $\mathbb{R}^{N}$, $N\geq 2$ are established. Finally, the behavior of sole blow-up point is further considered.
\end{abstract}
\maketitle
\section{Introduction and main results}

As we know, Liouville equation was firstly studied in 1853 in \cite{JLiouville}.
This equation has many applications in geometry and physical problem, for instance,
in the problem of prescribing Gaussian curvature \cite{SChangPYang,WChenCLi2}, the mean field equation
\cite{DingJostLiWang1,DingJostLiWang2,CChenCLin,ZDjadli,CLin} and the Chern-Simons model \cite{JSpruckYang,MStruweGTarantello,GTarantello}. About the blow-up analysis of solution of this
equation, we must mention the celebrated paper of Brezis and Merle \cite{BrezisMerle}. They firstly researched
the blow-up behavior and uniformly estimates for solutions of Liouville equation
\begin{equation}
-\Delta u=V(x)e^{u(x)}
\end{equation}
in bounded domain $\Omega\subset\mathbb{R}^{2}$ with $V(x)\in L^{p}(\Omega)$ and $e^{u}\in L^{p'}(\Omega)$
for some $1 < p\leq \infty$ and $p'$ is the H\"{o}lder
conjugate index of $p$. From then, blow-up analysis about the solutions of various
equation and system of equations have been extensively established (see \cite{BartolucciTarantello,DingJostLiWang1,JJostGWang,Jliyli,YLi2,RenWei} and their references). For example, Ren and Wei in \cite{RenWei} established
similar results for $N$-Laplacian operators in $\Omega\subset \mathbb{R}^{N}$, $N\geq 2$. Recently, Wang and Xia \cite{WangCXia1} generalized the blow-up analysis for Liouville type equation with anisotropic Laplacian (or Finsler-Laplacian)
in $\Omega\subset \mathbb{R}^{2}$. At the same time, they also defined $N$-anisotropic Laplacian (or $N$-Finsler Laplacian) as follows:
\begin{equation}
-Q_{N} u:=\sum_{i=1}^{N}\frac{\partial}{\partial x_{i}}(F^{N-1}(\nabla u)F_{\xi_i}(\nabla u)),
\end{equation}
where $F\in C^{2}(\mathbb{R}^{N}\backslash\{0\})$ is a convex and homogeneous function (see Section 2) and
$F_{\xi_{i}}=\frac{\partial F}{\partial \xi_{i}}$.
The anisotropic Laplacian is closely related to a smooth, convex
hypersurface in $\mathbb{R}^{N}$, which is called the Wulff shape. This kind of operator was initiated study in
Wulff's paper (see \cite{GWulff}). Recently, this operator has been widely studied by many mathematicians, see \cite{AlvinoFeroneTrombetti,BelloniFerone,VFeroneBKawohl,IFonsecaSMuller,GWangCXia2,GWangCXia3} and references therein.

It is natural to ask whether it has the similar Brezis-Merle result for Liouville
type equation with $N$-anisotropic Laplacian in bounded domain
in higher dimensions. This paper give the infirmative answer to this question.
More precisely, for a bounded domain $\Omega\subset \mathbb{R}^{N}$, $N\geq 2$, we consider the following quasilinear
equation
\begin{equation}
-Q_{N} u=V(x)e^{u(x)}.
\end{equation}
Equation (1.3) is also called $N$-anisotropic (or $N$-Finsler) Liouville equation.

 It is well known that in the isotropic case, i.e. $F(\xi) = |\xi|$, when $N=2$, $Q_N$ is the ordinary Laplacian operator; when $N>2$, $Q_N$ is the $N$-Laplacian operator. In the anisotropic case, when $N=2$, $Q_{N}$ is  anisotropic Laplacian operator. Therefore, the results of this paper extend that of ordinary Laplacian operator, $N$-Laplacian operator and anisotropic Laplacian operator.

In this paper, we firstly obtain some properties for $Q_{N}$, i.e. the weak maximal principle, the comparison principle and the mean value property. Then we get an a priori estimate for solutions of $N$-anisotropic Liouville equation, i.e. Theorem 1.1 and 1.2 in the paper. Because of the nonlinearity of $Q_{N}$, we have no concrete Green representation formula. Therefore, to prove Theorem 1.1 and 1.2, we use the level set method in \cite{RenWei} and convex symmetrization technique in \cite{AlvinoFeroneTrombetti} to conquer this difficult. From this, we research the blow-up behavior about this equation in bounded domain in $\mathbb{R}^{N}$, $N\geq 2$, which is the content of Theorem 1.3. In Theorem 1.4, the behavior of solo blow-up point is further studied. In the isotropic case, Li in \cite{YLi1} obtain this result using the the method of moving plane. For the anisotropic case, one have no idea how to use this method. Here we obtain this result by analyzing the Pohozaev identity and using the expansion of Green function of $N$-anisotropic Laplacian operator.

In addition, it is worthy mentioning that the positive definiteness of $Hess(F^{N})$, $N\geq2$ is needed in proving the main results of this paper, but condition on $F$ is that $Hess(F^{2})$ is positive definite. In fact, one can deduce positive definiteness of $Hess(F^{N})$ from that of $Hess(F^{2})$, which can be find in the proof of Theorem 3.2 in Section 3.

The main results of this paper are stated as follows.
\begin{thm}\label{thm-main1.1}
Suppose that $\Omega$ is a bounded domain in $\mathbb{R}^{N}$, $N\geq 2$ and $u$ is a weak solution of
\begin{eqnarray}
\begin{cases}
 -Q_{N}u=f(x)& \text{in}\ \Omega,\\
u|_{\partial\Omega}=0.
\end{cases}
\end{eqnarray}
Then for $f\in L^{1}(\Omega)$ and for any $\delta\in (0,N^{\frac{N}{N-1}}k^{\frac{1}{N-1}})=(0,\beta_{N})$, it follows that
\begin{eqnarray}
\int_{\Omega}\exp\left\{\frac{(\beta_{N}-\delta)|u(x)|}{||f||_{L^{1}(\Omega)}^{1/(N-1)}}\right\}dx\leq \frac{\beta_{N}}{\delta}|\Omega|.
\end{eqnarray}
\end{thm}

\begin{thm}\label{thm-main1.2}
Suppose that $u$ and $v$ are the weak solution of
\begin{eqnarray}
-Q_{N}u=f(x)>0\ \text{in} \ \ \Omega
\end{eqnarray}
and
\begin{eqnarray}
\begin{cases}
 -Q_{N}v=0& \text{in}\ \Omega,\\
v|_{\partial\Omega}=u,
\end{cases}
\end{eqnarray}
respectively. Then for any $\delta\in (0,\beta_{N})$, we have
\begin{eqnarray}
\int_{\Omega}\exp\left\{\frac{(\beta_{N}-\delta)d_{0}^{\frac{1}{N-1}}(|u-v|)}
{||f||_{L^{1}(\Omega)}^{\frac{1}{N-1}}}\right\}\leq \frac{|\Omega|}{\delta},
\end{eqnarray}
where
$$d_{0}=\inf\{d_{X,Y}:X,Y\in \mathbb{R}^{N}, X\neq 0, Y\neq 0,X\neq Y\}$$ and $$d_{X,Y}=\frac{<F^{N-1}(X)F_{\xi}(X)-F^{N-1}(Y)F_{\xi}(Y),X-Y>}{F^{N}(X-Y)}.$$
\end{thm}

\begin{thm}\label{thm-main1.3}
Suppose that $\Omega\subset\mathbb{R}^{N}$, $N\geq 2$ is a bounded domain and $u_{n}$ is a sequence of weak solutions of
\begin{eqnarray}
-Q_{N}u_n=V_{n}(x)e^{u_n}& \text{in}\ \Omega,
\end{eqnarray}
where $V_{n}(x)\geq 0$, $||V_{n}||_{L^{q}}\leq C_{1}$ for some $1<q\leq\infty $ and $||e^{u_n}||_{L^{q'}}\leq C_{2}$. Then one of the following possibilities happens (after taking subsequences):

(i) $u_n$ is bounded in $L^{\infty}_{\text{loc}}(\Omega)$;

(ii) $u_n\rightarrow -\infty$ uniformly on any compact subset of $\Omega$;

(iii) $S=\{p_{1},\cdots,p_{m}\}$ is a nonempty, finite set and $u_n\rightarrow -\infty$ uniformly on any compact subset of $\Omega\backslash S$. In addition, $V_{n}e^{u_n}\rightharpoonup \sum_{i=1}^{m}\alpha_{i}\delta_{p_{i}}$ in the sense of measures on $\Omega$ with $\alpha_{i}\geq (\frac{\beta_{N}}{q'})^{N-1}d_{0}$ for any $i$,
where  the blow-up set $S$ is defined by
$$S:=\{x\in\Omega: \exists\ x_{n}\in \Omega\ \text{such that}\ x_{n}\rightarrow x \ \text{and}\ u_{n}(x_{n})\rightarrow +\infty\}.$$
\end{thm}

\begin{thm}\label{thm-main1.4}
Suppose that $\Omega\subset\mathbb{R}^{N}$, $N\geq 2$ is a bounded domain and $u_{n}$ is a sequence of weak solutions of (1.9) with $$\int_{\Omega}e^{u_n}\leq C.$$
Let $(V_{n})$ be a sequences of Lipschitz continuous functions satisfying
\begin{eqnarray}
V_{n}(x)\geq 0, V_{n}(x)\rightarrow V \text{uniformly in}\ C_{0}(\overline{\Omega}),\ ||\nabla V_{n}||_{L^{\infty}}\leq C.
\end{eqnarray}
In addition, suppose that
\begin{eqnarray}
\max_{\partial\Omega}(u_{n})-\min_{\partial\Omega}(u_{n})\leq C.
\end{eqnarray}
Then if blow-up happens only at one point, the blow-up value $\alpha=\left(\frac{N^{N+1}k^{\frac{1}{N-1}}}{N-1}\right)^{\frac{1}{N-1}}$.
\end{thm}

This paper is organized as follows. In Section 2, we sum up the properties related to $F$ and recall some lemmas which will be used in the proof of main results. In section 3, we give the weak maximal principle, the comparison principle and the mean value property for $Q_{N}$. Brezis-Merle type concentraton-compactness formula and a priori estimate is arranged in Section 4. In Section 5 and Section 6, we give the proof of blow-up Theorem and Theorem 1.4 respectively.

\section{Preliminaries}

In this section, let us recall some concepts and properties related to $F$.

Let $F:\ \mathbb{R}^{N}\rightarrow [0,\infty)$ be a convex function of $C^{2}(\mathbb{R}^{N}\backslash \{0\})$, which is even and positively homogeneous of degree $1$, i.e. for any $t\in \mathbb{R}, \xi \in \mathbb{R}^{N}$, \begin{equation}
F(t\xi)=|t|F(\xi).
\end{equation}

 We also assume that $F(\xi)>0$ for any $\xi \neq0$, and $Hess(F^{2})$ is positive definite in $\mathbb{R}^{N}\backslash \{0\}$.
With the help of homogeneity of $F$, there exist two constant $0<a\leq b <\infty$, such that
\begin{equation}
a|\xi|\leq F(\xi)\leq b|\xi| \  \text {for any}\ \xi\in \mathbb{R}^{N}.\end{equation}

Consider the map $\Phi:\ \mathbb{S}^{N-1}\rightarrow \mathbb{R}^{N},\ \ \Phi(\xi)=F_{\xi}(\xi)$. Its image $\Phi(\mathbb{S}^{N-1})$ is a smooth, convex hypersurface in $\mathbb{R}^{N}$, which is called the Wulff shape of $F$.

If one defines the support function of $F$ as $F^{0}(x):=\sup _{\xi\in K}<x,\xi>$, where $K:=\{x\in \mathbb{R}^{N}:
 F(x)<1\}$, then it is easy to prove that $F^{0}: \mathbb{R}^{N}\rightarrow [0,\infty)$ is also a conves, homogeneous function and $F, F^{0}$ are polar to each other in the sense that
 $$F^{0}(x)=\sup_{\xi\neq 0}\frac{<x,\xi>}{F(x)}\ \text {and}\  F(\xi)=\sup_{x\neq 0}\frac{<x,\xi>}{F^{0}(x)}.$$

Let $\mathcal{W}_{F}:=\{x\in \mathbb{R}^{N}: F^{0}(x)\leq 1\}$ and $k=k_{N}=|\mathcal{W}_{F}|$, which is the Lebesgue measure of $\mathcal{W}_{F}$. Also, denote $\mathcal{W}_{r}(x_{0})$ by the Wulff ball of center at $x_{0}$ with radius $r$, i.e. $\mathcal{W}_{r}(x_{0})=\{x\in \mathbb{R}^{N}: F^{0}(x-x_{0})\leq r\}$.

  Next, we summarize the properties on $F$ and $F^{0}$, which can be proved easily by the assumption on $F$. See \cite{BellettiniPaolini,VFeroneBKawohl,WangCXia1}.

\begin{prop}\label{prop:eg}

We have the following properties:

(1)\quad$|F(x)-F(y)|\leq F(x+y)\leq F(x)+F(y)$;

(2)\quad$|\nabla F(x)|\leq C$\ for any $x\neq 0$;

(3)\quad $<\xi,\nabla F(\xi)>=F(\xi),\ <x,\nabla F^{0}(x)>=F^{0}(x)$ for any $x\neq 0,\ \xi\neq 0$;

(4)\quad$F(\nabla F^{0}(x))=1,\ F^{0}(\nabla F(\xi))=1$;

(5)\quad$ F_{\xi_{i}}(t\xi)=\text{sgn} (t) F_{\xi_{i}}(\xi)$;

(6) \quad $F^{0}(x)F_{\xi}(\nabla F^{0}(x))=x_{i}$.
\end{prop}

Next we give a co-area formula and isoperimetric inequality in the anisotropic situation. One can refer to \cite{AlvinoFeroneTrombetti,IFonsecaSMuller}.

For a bounded domain $\Omega\subset \mathbb{R}^{N}$ and a function of bounded variation $u\in BV(\Omega)$, denote the anisotropic bounded variation of $u$ with respect to $F$ by
$$\int_{\Omega}|\nabla u|_{F}=\sup\left\{\int_{\Omega}u\text{div}\sigma dx: \sigma \in C_{0}^{1}(\Omega),F^{0}(\sigma)\leq 1\right\},$$
and anisotropic perimeter of $E$  with respect to $F$ by
$$P_{F}(E):=\int_{\Omega}|\nabla \chi_{E}|_{F},$$
 where $E$ is a subset of $\Omega$ and $\chi_{E}$ is the characeristic function of $E$.
The co-area formula and isoperimetric inequality can be expressed by
\begin{equation}
\int_{\Omega}|\nabla u|_{F}=\int_{0}^{\infty}P_{F}(|u|>t)dt,
\end{equation}
and
\begin{equation}
P_{F}(E)\geq Nk^{\frac{1}{N}}|E|^{1-\frac{1}{N}}
\end{equation}
respectively.

Let us review the convex symmetrization which is the generalization of the Schwarz symmetrization (see \cite{GTalenti}). The one-dimensional decreasing rearrangement of $u$ is
$$u^{*}(t)=\sup \{s\geq 0: |\{x\in \Omega: |u(x)|>s\}|>t\},$$
for $t\in \mathbb{R}$. The convex symmetrization of $u$ is defined as
$$u^{\star}(x)=u^{*}(kF^{0}(x)^{N}), \ \text{for}\  x\in \Omega^{\star},$$
where $\Omega^{\star}$ is the homothetic Wulff ball centered at the origin having the same measure as $\Omega$.

Next, let us give some lemmas which will be important in proving main results of this paper.
\begin{lem}\label{lem2.1}
(see \cite{AlvinoFeroneTrombetti}) Let $u\in W_{0}^{1,p}(\Omega)$ for $p\geq 1$. Then $u^{\star}\in W_{0}^{1,p}(\Omega^{\star})$ and
$$\int_{\Omega^{\star}}F^{p}(\nabla u^{\star})dx\leq\int_{\Omega}F^{p}(\nabla u)dx .$$
\end{lem}

\begin{lem}\label{lem2.2}
(see  \cite{AlvinoFeroneTrombetti}) Let $u,v\in W_{0}^{1,2}(\Omega)$ be the weak solutions of the following equations
\begin{eqnarray}
\begin{cases}
 -\text{div}(a(x,u,\nabla u))=f(x)& \text{in}\ \Omega,\\
u|_{\partial\Omega}=0,
\end{cases}
\end{eqnarray}
and

\begin{eqnarray}\label{2.6}
\begin{cases}
 -Q_{N}v=f^{\star}(x)& \text{in}\ \Omega^{\star},\\
v|_{\partial\Omega^{\star}}=0,
\end{cases}
\end{eqnarray}
respectively, where $f\in L^{\frac{2N}{N+2}}(\Omega)$ if $N\geq 3$ or $f\in L^{p}(\Omega)$, $p>1$ if $N=2$.
Suppose that $a(x,\eta,\xi)$ is vector-value Carath\'{e}odory function satisfying
$$<a(x,\eta,\xi),\xi>\geq F^{2}(\xi) \ a.e.\ x\in\Omega,\ \eta\in \mathbb{R},\ \xi\in \mathbb{R}^{N}.$$
Then it follows that
$$u^{\star}\leq v\  \text{in}\ \Omega^{\star}.$$
\end{lem}

\begin{lem}\label{lem2.3}
(see \cite{WangCXia1}) Assume u satisfies $-Q_{N}u=0$ in $\Omega\backslash\{0\}$ such that $u(x)/\Gamma(x)$
remains bounded in some neighborhood of 0. Then there exists a real number $\gamma$ and
$h\in C^{0}(\Omega)$ such that
\begin{equation}
u=\gamma \Gamma +h.
\end{equation}

Moreover, when $\gamma\neq0$, the following relation holds
\begin{equation}
\lim_{x\rightarrow 0}F^{0}(x)\nabla h(x)=0
\end{equation}
and u satisfies $-Q_{N}u=\gamma \delta_{0}$ in the sense of measures in $\Omega$, where
\begin{equation}
\Gamma(x)=-\frac{1}{(Nk)^{\frac{1}{N-1}}}\log F^{0}(x).
\end{equation}

\end{lem}

\begin{lem}\label{lem2.4}
 (see \cite{WangCXia1}) There exists a unique function $G(\cdot, 0)\in C^{1,\alpha}(\Omega\backslash \{0\})$ with
$|\nabla G|\in L^{1}(\Omega)$ and $G/\Gamma\in L^{\infty}(\Omega)$ satisfying
\begin{eqnarray}
\begin{cases}
 -Q_{N}G(\cdot,0)=\delta_{0}& \text{in}\ \Omega,\\
G(\cdot,0)|_{\partial\Omega}=0.
\end{cases}
\end{eqnarray}
Moreover, $G=\Gamma+h$ with $h\in C^{0}(\Omega)$ satisfying (2.8).
\end{lem}

\section{The properties of $Q_{N}$ }
In this section, we will give the weak maximum principle, weak comparison principle and mean value property for the $N$-anisotropic Laplacian operator $Q_{N}$.

\begin{thm}\label{thm3.1}
(Weak maximum principle) Suppose that $-Q_{N}u\leq 0$ in $\Omega$ and $u\leq M$ on $\partial \Omega$, then $u(x)$ attains its maximum on the boundary, i.e. $u\leq M$ a.e. in $\Omega$.
\end{thm}

\begin{proof}
Let $\Omega^{+}=\{x\in\Omega: u(x)>M\}.$ Multiplying $-Q_{N}u\leq 0$ by $(u-M)^{+}$ and using the integration by parts
, it follows that
\begin{eqnarray}
\begin{split}
 0\geq&\int_{\Omega^{+}}\sum_{i}F^{N-1}(\nabla u)F_{\xi_{i}}(\nabla u)u_{x_{i}}dx\\
 &\ \ +\int_{\partial \Omega^{+}}\sum_{i}F^{N-1}(\nabla u)F_{\xi_i}(\nabla u)(u-M)^{+}\nu_{i}ds\\
=&\int_{\Omega^{+}}(F(\nabla u))^{N}dx.
\end{split}
\end{eqnarray}
Thus $\Omega^{+}$ has measures zero and $u(x)\leq M$ a.e. in $\Omega$.
\end{proof}

\begin{thm}\label{thm3.2}
(Comparison principle) Suppose that $-Q_{N}u\leq -Q_{N}v$ in $\Omega$ and $u\leq v$ on $\partial \Omega$, then $u\leq v$ a.e. in $\Omega$.
\end{thm}

\begin{proof}
Let $\Omega^{+}=\{x\in\Omega: u(x)>v(x)\}.$ Assume that $\Omega^{+}$ has positive measure. Multiplying $-Q_{N}u\leq -Q_{N}v$ by $(u-v)^{+}$, it follows that
\begin{eqnarray}
\begin{split}
0\geq&\int_{\Omega^{+}}[F^{N-1}(\nabla u)F_{\xi}(\nabla u)-F^{N-1}(\nabla v)F_{\xi}(\nabla v)](\nabla u-\nabla v)dx\\
=&\int_{\Omega^{+}}\sum_{i,j}\left[\int_{0}^{1}\left(\frac{1}{N}F^{N}_{\xi_{i}\xi_{j}}
(t\nabla u+(1-t)\nabla u_{1})\right)dt\right]\partial_{x_{j}}(u-v)\partial_{x_{i}}(u-v)dx.
\end{split}
\end{eqnarray}
Since $$\nabla(F^{N})=\nabla((F^{2})^{\frac{N}{2}})=\frac{N}{2}(F^{2})^{\frac{N}{2}-1}\nabla(F^{2}),$$
and
$$\nabla^{2}(F^{N})=\frac{N}{2}(\frac{N}{2}-1)(F^{2})^{\frac{N}{2}-2}[\nabla(F^{2})]^{T}\nabla(F^{2})
+\frac{N}{2}(F^{2})^{\frac{N}{2}-1}\nabla^{2}(F^{2}),$$
where $[A]^{T}$ denotes the transpose of matrix $A$. Then from the positive definiteness of $Hess(F^{2})$, we obtain that $Hess(F^{N})$, $N\geq 2$ is also positive definite.
Thus we have that $\nabla u=\nabla v$ a.e. in $\Omega^{+}$. Since $u=v$ on $\partial\Omega^{+}$, it is easy to see that $\Omega^{+}$ has measures zero and $u\leq v$ a.e. in $\Omega$.
\end{proof}

\begin{thm}\label{thm3.3}
(Mean value property) Suppose that $F$ and $F^{0}$ satisfy
\begin{eqnarray}\label{cond3.3}
<F_{\xi}(x),F^{0}_{\xi}(y)>=\frac{<x,y>}{F^{N-1}(x)F^{0}(y)}
 \end{eqnarray}
 for all $x,y\in \mathbb{R}^{N}$.
If $Q_{N}u=0$ in $\Omega$ and $\mathcal{W}_{\rho}(0)=\{x\in \mathbb{R}^{N}: F^{0}(x)<\rho\}\subset \Omega$, then for every ball of radius $r\in (0,\rho)$, $u$ satisfies the mean value property on spheres,
\begin{eqnarray}
u(0)=\frac{1}{|\partial \mathcal{W}_{1}(0)|r^{N-1}}\int_{\partial \mathcal{W}_{r}(0)} u(x)ds
 \end{eqnarray}
and the corresponding mean value property on balls
\begin{eqnarray}
u(0)=\frac{1}{kr^{N}}\int_{\mathcal{W}_{r}(0)} u(x)dx.
 \end{eqnarray}
\end{thm}

\begin{proof}
Let
\begin{eqnarray}
\begin{split}
\phi(r):=&\frac{1}{|\partial \mathcal{W}_{1}(0)|r^{N-1}}\int_{\partial \mathcal{W}_{r}(0)}u(x)ds(x)\\
=&\frac{1}{|\partial \mathcal{W}_{1}(0)|}\int_{\partial \mathcal{W}_{1}(0)}u(rz)ds(z).
\end{split}
\end{eqnarray}
Then
\begin{eqnarray}
\begin{split}
\phi'(r)&=\frac{1}{|\partial \mathcal{W}_{1}(0)|}\int_{\partial \mathcal{W}_{1}(0)}<\nabla u(rz),z>ds(z)\\
&=\frac{1}{|\partial \mathcal{W}_{1}(0)|r^{N-1}}\int_{\partial \mathcal{W}_{r}(0)}<\nabla u(x),\frac{x}{r}>ds(x).
\end{split}
\end{eqnarray}
By the condition \eqref{cond3.3}, it follows that
\begin{eqnarray}
<\nabla u,x>=F^{N-1}(\nabla u)<F_{\xi}(\nabla u), F^{0}_{\xi}(x)>F^{0}(x).
\end{eqnarray}
Since $F^{0}=r$ and $\nu=F^{0}_{\xi}(x)$ on $\partial \mathcal{W}_{r}(0)$, thus by integration by parts we have
\begin{eqnarray}
\phi'(r)=\int_{\partial \mathcal{W}_{r}(0)}\sum_{i=1}^{N}F^{N-1}(\nabla u)F_{\xi}(\nabla u)\nu_{i}ds=\int_{\mathcal{W}_{r}(0)}Q_{N}uds=0.
\end{eqnarray}
Thus the proof of the mean value property on spheres is completed. Integrating with respect to $r$, one can get the mean value property on ball.
\end{proof}

\section{A Priori Estimates}

In this section, with the method of level set and convex symmetrization, we firstly prove Brezis-Merle type concentration-compactness formula, i.e. Theorem 1.1 and Theorem 1.2. At the same time, we obtain the $L^{\infty}$ estimates for a single solution and uniform $L^{\infty}$  bounds for solutions of $-Q_{N}u=Ve^{u}$.

\begin{proof}[Proof of Theorem 1.1]
Assume that $R>0$ is the constant such that $|\Omega|=kR^{N}$. Let $v$ be the unique solution of the convex symmetrized Dirichlet problem \eqref{2.6}. According to Lemma 2.3, it follows that
$$u^{*}\leq v,$$
where $u^{*}$ is the convex symmetric decreasing rearrangement of $u$. In addition, $v(x)=v(F^{0}(x))=v(r)$ is convex symmetric with respect to $F$ and satisfies the following equation:
\begin{eqnarray}
\begin{cases}
 (|v'|^{N-2}v')'+\frac{N-1}{r}(|v'|^{N-2}v')+f^{\star}(r)=0,\\
v'(0)=0, v(R)=0.
\end{cases}
\end{eqnarray}
It follows that
$$-v'(r)=\frac{\left(\int_{0}^{r}s^{N-1}f^{\star}(s)ds\right)^{\frac{1}{N-1}}}{r}\leq \frac{1}{(Nk)^{\frac{1}{N-1}}}\frac{1}{r}||f^{\star}||_{L^{1}(\Omega^{\star})}^{\frac{1}{N-1}}
=\frac{1}{(Nk)^{\frac{1}{N-1}}}\frac{1}{r}||f||_{L^{1}(\Omega)}^{\frac{1}{N-1}}.$$
Then
$$v(r)=-\int_{r}^{R}v'(t)dt\leq \int_{r}^{R}\frac{1}{(Nk)^{\frac{1}{N-1}}}\frac{1}{r}||f||_{L^{1}(\Omega)}^{\frac{1}{N-1}}dt
=\frac{1}{(Nk)^{\frac{1}{N-1}}}||f||_{L^{1}(\Omega)}^{\frac{1}{N-1}}\log\frac{R}{r}.$$
Therefore
\begin{eqnarray}
\begin{split}
&\int_{\Omega}\exp\left(\frac{(N-\epsilon)(Nk)^{\frac{1}{N-1}}|u(x)|}
{||f||_{L^{1}(\Omega)}^{\frac{1}{N-1}}}\right)dx\\
=&\int_{\Omega^{\star}}\exp\left(\frac{(N-\epsilon)(Nk)^{\frac{1}{N-1}}u^{\star}(x)}
{||f||_{L^{1}(\Omega)}^{\frac{1}{N-1}}}\right)dx\\
\leq&\int_{\Omega^{\star}}\exp\left(\frac{(N-\epsilon)(Nk)^{\frac{1}{N-1}}v(x)}
{||f||_{L^{1}(\Omega)}^{\frac{1}{N-1}}}\right)dx\\
=&\int_{0}^{R}Nkr^{N-1}\exp\log\left(\frac{R}{r}\right)^{N-\epsilon}dr=\frac{Nk}{\epsilon}R^{N}.
\end{split}
\end{eqnarray}
Let $\epsilon(Nk)^{\frac{1}{N-1}}=\delta$, one have
\begin{eqnarray}
\int_{\Omega}\exp\left\{\frac{(\beta_{N}-\delta)|u(x)|}{||f||_{L^{1}}^{1/(N-1)}}\right\}dx\leq \frac{\beta_{N}}{\delta}kR^{N}=\frac{\beta_{N}}{\delta}|\Omega|.
\end{eqnarray}
\end{proof}

\begin{cor}\label{2.1}
Suppose that $u$ is a weak solution of (1.4) with $f\in L^{1}(\Omega)$. Then for any constant $s>0$, we have $\exp(s|u|)\in L^{1}(\Omega)$.
\end{cor}

\begin{proof}\label{proof of cor 2.1}
Let $0<\epsilon<\frac{1}{s}$, we split $f$ as $f=f_{1}+f_{2}$ with $||f_{1}||_{L^{1}(\Omega)}<\epsilon$
and $f_{2}\in L^{\infty}(\Omega)$. Assume that $u_{1}$ is the solution of
\begin{eqnarray}
\begin{cases}
 -Q_{N}u_{1}=f_{1}(x)& \text{in}\ \Omega,\\
u_{1}|_{\partial\Omega}=0.
\end{cases}
\end{eqnarray}
Choosing $\delta=\beta_{N}-1$ in Theorem 1.1, it follows that $\int_{\Omega}\exp\left[\frac{|u_{1}|}{||f_{1}||_{L^{1}(\Omega)}}\right]<\infty$ and thus $\int_{\Omega}\exp[s|u_{1}|]<\infty$. In addition, by the Fundamental Theorem of Calculus, we have
$$f_{2}=-(Q_{N}u-Q_{N}u_{1})=-\widetilde{Q}_{N}(u-u_{1}),$$
where
$$\widetilde{Q}_{N}(u-u_{1})=\sum_{i,j}\partial_{x_{i}}\left[\int_{0}^{1}\left(\frac{1}{N}F^{N}_{\xi_{i}\xi_{j}}
(t\nabla u+(1-t)\nabla u_{1})\right)dt\partial_{x_{j}}(u-u_{1})\right].$$
As in the proof of Theorem 3.2, from $Hess(F^{2})$ is positive definite, we get that $Hess(F^{N})$, $N\geq 2$ is also positive definite. It is easy to see that $\widetilde{Q}_{N}$ is an elliptic operator. From standard elliptic theory, we easily obtain $||u-u_{1}||_{L^{\infty}}\leq C||f_{2}||_{L^{\infty}}$. The desired conclusion is followed.
\end{proof}

\begin{cor}\label{2.2}
Suppose that $u$ is a weak solution of
\begin{eqnarray}
\begin{cases}
 -Q_{N}u=V(x)e^{u}& \text{in}\ \Omega,\\
u|_{\partial\Omega}=0,
\end{cases}
\end{eqnarray}
 where $V(x)\in L^{q}(\Omega)$ and $e^{u}\in L^{q'}(\Omega)$ for some $1<q\leq\infty$. Then $u\in L^{\infty}(\Omega)$.
\end{cor}

\begin{proof}\label{proof of cor 2.2}
By Corollary 4.1 and the standard elliptic theory for quasilinear equation, one can obtain the desired result.
\end{proof}

\begin{cor}\label{2.3}
Suppose that $u_{n}$ is a weak solution of (1.9) with $u_{n}=0$ on $\partial\Omega$,
where $$||V_{n}||_{L^{q}(\Omega)}\ \text{for some}\ 1<q\leq \infty,$$
and
$$\int_{\Omega}|V_{n}|e^{u_{n}}\leq \epsilon_{0}<\frac{\beta_{N}}{q'}.$$
Then $$||u_{n}||_{L^{\infty}(\Omega)}\leq C.$$
\end{cor}

\begin{proof}\label{proof of cor 2.3}
Choosing $\delta>0$ such that $\beta_{N}-\delta>\epsilon_{0}(q'+\delta)$. By Theorem 1.1, we have
$e^{u_{n}}$ is bounded in $L^{q'+\delta}(\Omega)$ and $V_{n}e^{u_{n}}$ is bounded in $L^{p}(\Omega)$ for some $p>1$.
Then the conclusion can be obtained by the standard Morse iteration method.
\end{proof}

\begin{proof}[Proof of Theorem 1.2]
Let $\Omega_{t}=\{x\in \Omega:|u-v|>t\}$ and $\mu(t)=|\Omega_{t}|$. It follows that
\begin{eqnarray}
\begin{split}
&-(Q_{N}(u)-Q_{N}(v))=-\widetilde{Q}_{N}(u-v)\\
=&-\sum_{i,j}\partial_{x_{i}}\left[\int_{0}^{1}\left(\frac{1}{N}F^{N}_{\xi_{i}\xi_{j}}
((1-t)\nabla u+t\nabla v)\right)dt\partial_{x_{j}}(u-v)\right]=f>0.
\end{split}
\end{eqnarray}
Since this equation is uniformly elliptic, applying Hopf's boundary lemma, one can deduce that
$$\partial_{\nu}(u-v)<0,\ \ \nabla u-\nabla v\neq0\ \text{on}\ \partial\Omega_{t}.$$
Then
\begin{eqnarray}
\begin{split}
\int_{\Omega_{t}}f(x)dx&=\int_{\Omega_{t}}-(Q_{N}(u)-Q_{N}(v))dx\\
&=\int_{\partial\Omega_{t}}\left<F^{N-1}(\nabla u)F_{\xi}(\nabla u)-F^{N-1}(\nabla v)F_{\xi}(\nabla v),
\frac{\nabla u-\nabla v}{|\nabla u-\nabla v|}\right>\\
&\geq d_{0}\int_{\partial\Omega_{t}}\frac{F^{N}(\nabla u-\nabla v)}{|\nabla u-\nabla v|}.
\end{split}
\end{eqnarray}
By the isoperimetric inequality, the co-area formula and the H\"{o}lder's inequality, it follows that
\begin{eqnarray}
\begin{split}
Nk^{\frac{1}{N}}\mu(t)^{\frac{N-1}{N}}
\leq P_{F}(\Omega_{t})&=-\frac{d}{dt}\int_{\Omega_{t}}F(\nabla u-\nabla v)dx\\
&=\int_{\partial\Omega_{t}}\frac{F(\nabla u-\nabla v)}{|\nabla u-\nabla v|}\\
&\leq\left(\int_{\partial\Omega_{t}}\frac{F^{N}(\nabla u-\nabla v)}{|\nabla u-\nabla v|}\right)^{\frac{1}{N}}
\left(\int_{\partial\Omega_{t}}\frac{1}{|\nabla u-\nabla v|}\right)^{\frac{N-1}{N}}\\
&=\left(\int_{\partial\Omega_{t}}\frac{F^{N}(\nabla u-\nabla v)}{|\nabla u-\nabla v|}\right)^{\frac{1}{N}}\left(-\mu'(t)\right)^{\frac{N-1}{N}}.
\end{split}
\end{eqnarray}
It follows that
$$-\mu'(t)\geq \frac{d_{0}^{\frac{1}{N-1}}N^{\frac{N}{N-1}}k^{\frac{1}{N-1}}\mu(t)}{(\int_{\Omega_{t}}f(x)dx)^{\frac{1}{N-1}}}.$$
Therefore
\begin{eqnarray}\label{4.9}
-\frac{dt}{d\mu}\leq \frac{||f||_{L^{1}(\Omega)}^{\frac{1}{N-1}}}{d_{0}^{\frac{1}{N-1}}N^{\frac{N}{N-1}}k^{\frac{1}{N-1}}\mu(t)}.
\end{eqnarray}
Integrating \eqref{4.9} over $(\mu,|\Omega|)$, we obtain
$$t(\mu)\leq \frac{||f||_{L^{1}(\Omega)}^{\frac{1}{N-1}}}{d_{0}^{\frac{1}{N-1}}N^{\frac{N}{N-1}}k^{\frac{1}{N-1}}}
\log(\frac{|\Omega|}{\mu}).$$
It is easy to see that
$$\exp\left(\frac{(1-\epsilon)d_{0}^{\frac{1}{N-1}}N^{\frac{N}{N-1}}k^{\frac{1}{N-1}}}
{||f||_{L^{1}(\Omega)}^{\frac{1}{N-1}}}t(\mu)\right)\leq (\frac{|\Omega|}{\mu})^{1-\epsilon}.$$
Thus one can get
\begin{eqnarray}
\begin{split}
&\int_{\Omega}\exp\left(\frac{(1-\epsilon)d_{0}^{\frac{1}{N-1}}N^{\frac{N}{N-1}}k^{\frac{1}{N-1}}}
{||f||_{L^{1}(\Omega)}^{\frac{1}{N-1}}}|u-v|\right)dx\\
=&\int_{0}^{\infty}\exp\left(\frac{(1-\epsilon)d_{0}^{\frac{1}{N-1}}N^{\frac{N}{N-1}}k^{\frac{1}{N-1}}t}
{||f||_{L^{1}(\Omega)}^{\frac{1}{N-1}}})\right)(-\mu'(t)dt\\
=&\int_{0}^{|\Omega|}\exp\left(\frac{(1-\epsilon)d_{0}^{\frac{1}{N-1}}N^{\frac{N}{N-1}}k^{\frac{1}{N-1}}t(\mu)}
{||f||_{L^{1}(\Omega)}^{\frac{1}{N-1}}}\right)d\mu\leq\frac{|\Omega|}{\epsilon}.
\end{split}
\end{eqnarray}
This means that
$$\int_{\Omega}\exp\left(\frac{(1-\epsilon)d_{0}^{\frac{1}{N-1}}N^{\frac{1}{N-1}}k^{\frac{1}{N-1}}}
{||f||_{L^{1}(\Omega)}^{\frac{1}{N-1}}}|u-v|\right)dx\leq\frac{|\Omega|}{\epsilon}.$$
Choosing $\delta=\epsilon N^{\frac{N}{N-1}}k^{\frac{1}{N-1}}$, one can easily obtain the desired result.
\end{proof}

\begin{cor}\label{3.1}
Suppose that $u_{n}$ is a weak solution of (1.9),
where $V_{n}\geq 0\ \text{in}\ \Omega$. For $B_{R}\subset \Omega$, $ ||V_{n}||_{L^{q}(B_{R})}\leq C_{1}\ \text{for some}\ 1<q\leq \infty$, $||u_{n}^{+}||_{L^{N}(B_{R})}\leq C_{2}$
and
\begin{equation}\label{smallcondition}
\int_{B_{R}}V_{n}e^{u_{n}}\leq \epsilon_{0}<(\frac{\beta_{N}}{q'})^{N-1}d_{0}.
\end{equation}
Then there exists a positive constant $C$ such that
$$||u_{n}^{+}||_{L_{\text{loc}}^{\infty}(\Omega)}\leq C.$$
\end{cor}

\begin{proof}
Consider the problem
\begin{eqnarray}
\begin{cases}
 -Q_{N}v_{n}=0& \text{in}\ B_{R},\\
v_{n}=u_{n}^{+}&\ \text{on}\ \partial B_{R}.
\end{cases}
\end{eqnarray}
Then, by the weak comparison principle Theorem 3.2, one implies that $v_{n}\leq u_{n}$ in $B_{R}$, that is, $v_{n}^{+}\leq u^{+}_{n}$.
Therefore, $$||v_{n}^{+}||_{L^{N}(B_{R})}\leq ||u_{n}^{+}||_{L^{N}(B_{R})}\leq C_{2}.$$
By Serrin's local a priori estimates (see \cite{JSerrin}), we obtain that
$||v_{n}^{+}||_{L^{\infty}(B_{R})}\leq C$.

By Theorem 1.2, we know
$$\int_{B_{R}}\exp\left\{\frac{(\beta_{N}-\delta)d_{0}^{\frac{1}{N-1}}(u_{n}-v_{n})}
{||V_{n}e^{u_{n}}||_{L^{1}(B_{R})}^{\frac{1}{N-1}}}\right\}\leq \frac{|B_{R}|}{\delta}.$$
Note that $v_{n}\leq u_{n}$ implies $u_{n}^{+}-v_{n}^{+}\leq u_{n}-v_{n}$, then
$$\int_{B_{R/2}}\exp\left\{\frac{(\beta_{N}-\delta)d_{0}^{\frac{1}{N-1}}(u_{n}^{+}-v_{n}^{+})}
{||V_{n}e^{u_{n}}||_{L^{1}(B_{R})}^{\frac{1}{N-1}}}\right\}\leq \frac{|B_{R}|}{\delta}.$$
Combining this with $||v_{n}^{+}||_{L^{\infty}(B_{R})}\leq C$ and the smallness condition \eqref{smallcondition}, we know
$$\int_{B_{R/2}}\exp\left\{\frac{(\beta_{N}-\delta)d_{0}^{\frac{1}{N-1}}u_{n}^{+}}
{\epsilon_{0}^{\frac{1}{N-1}}}\right\}
\leq\int_{B_{R/2}}\exp\left\{\frac{(\beta_{N}-\delta)d_{0}^{\frac{1}{N-1}}u_{n}^{+}}
{||V_{n}e^{u_{n}}||_{L^{1}(B_{R})}^{\frac{1}{N-1}}}\right\}\leq \frac{|B_{R}|}{\delta}.$$
Now, choosing $\delta$ such that $(\beta_{N}-\delta)d_{0}^{\frac{1}{N-1}}>{\epsilon_{0}^{\frac{1}{N-1}}}(q'+\delta)$, we
deduce that $e^{u_{n}^{+}}$ is bounded in $L^{q'+\delta}(B_{R/2})$, which implies that $V_{n}e^{u_{n}^{+}}$
is bounded in $L^{p}(B_{R/2})$ for some $p>1$. If we take the problem
\begin{eqnarray}
\begin{cases}
 -Q_{N}w_{n}=V_{n}e^{u_{n}^{+}}& \text{in}\ B_{R/2},\\
w_{n}=u_{n}^{+}& \text{on}\ \partial B_{R/2}.
\end{cases}
\end{eqnarray}
Using Serrin's local a priori estimate again, we obtain
$||w_{n}||_{L^{\infty}(B_{R/4})}\leq C$. Note that $u_{n}^{+}\leq w_{n}$ by the weak comparison principle,
then $||u_{n}^{+}||_{L^{\infty}(B_{R/4})}\leq C$.
\end{proof}

\section{Proof of Blow-up theorem}

Using the results of Section 4, let us prove the Blow-up Theorem, i.e. Theorem 1.3.

\begin{proof}[Proof of Theorem 1.3]
Since $V_{n}e^{u_{n}}$ is bounded in $L^{1}(\Omega)$, then there exists a nonnegative bounded measures $\mu$ such that for a subsequence (still denoted by $V_{n}e^{u_{n}}$),
$$\int V_{n}e^{u_{n}}\psi\rightarrow \int\psi d\mu$$
in the sense of measures on $\Omega$ for any $\psi\in C_{c}(\Omega)$.
One call that a point $x_{0}\in \Omega$ is a $\gamma$ regular point if for some $\gamma>0$, there is a function
$\psi\in C_{c}(\Omega)$, $0\leq\psi\leq1$ with $\psi=1$ in some neighborhood of $x_{0}$, such that
$$\int \psi d\mu<\gamma.$$
Let $\Sigma$ be the set of non $\gamma$ regular point in $\Omega$. It is easy to see that
$$x_{0}\in\Sigma\Leftrightarrow \mu(\{x_{0}\})\geq \gamma.$$
In addition, since $\mu$ is a bounded measure, one know that $\Sigma$ is finite.

 We now split the proof into three steps.

{ \bf  Step 1.} $S=\Sigma$.
If $x_{0}$ is a $(\frac{\beta_{N}}{q'})^{N-1}d_{0}$ regular point, then by Corollary 4.4, there is $R_{0}>0$ such that $u_{n}^{+}$ is bounded in $L^{\infty}(B_{R_{0}}(x_{0}))$. In other words, if $x_{0}\notin \Sigma$, then
$x_{0}\notin S$, i.e. $S\subset \Sigma$. Conversely, with the similar method in \cite{BrezisMerle}, one can get $\Sigma\subset S$.
Here we omit the details.

{ \bf  Step 2.} $S=\emptyset$ implies (i) or (ii) holds.
$S=\emptyset$ means that $u_{n}$ is bounded in $L^{\infty}_{\text{loc}}(\Omega)$ and $V_{n}e^{u_{n}}$ is bounded in $L^{q}_{\text{loc}}(\Omega)$. This implies that $\mu\in L^{1}(\Omega)\bigcap L^{q}_{\text{loc}}(\Omega)$.
Assume that $v_{n} $ is the weak solution of
\begin{eqnarray}
\begin{cases}
 -Q_{N}v_{n}=V_{n}e^{u_{n}}& \text{in}\ \Omega,\\
v_{n}=0& \text{on}\ \partial \Omega.
\end{cases}
\end{eqnarray}
Obviously, $v_{n}\rightarrow v$ uniformly on any compact subset of $\Omega$, where $v$ is the weak solution of
\begin{eqnarray}
\begin{cases}
 -Q_{N}v=\mu& \text{in}\ \Omega,\\
v=0& \text{on}\ \partial \Omega.
\end{cases}
\end{eqnarray}
Let $w_{n}=u_{n}-v_{n}$, then $-\widetilde{Q}_{N}w_{n}=0$ in $\Omega$ and $w_{n^{+}}$ is bounded in $L^{\infty}_{\text{loc}}(\Omega)$, where $\widetilde{Q}_{N}$ is defined as in Corollary 4.1. Using the Harnack's inequality (see \cite{JSerrin}), one obtains the following two possibility:

$(a)$ a subsequence $(w_{n_{k}})$ is bounded in $L^{\infty}_{\text{loc}}(\Omega)$; or

$(b)$ $(w_{n})$ converges uniformly to $-\infty$ on compact subset of $\Omega$.

It is easy to see that $(a)$ corresponds to Case (i) and $(b)$ to Case (ii).

{ \bf  Step 3.} $S\neq\emptyset$ implies (iii) holds.
Similar with Step 2, by Harnack's inequality \cite{JSerrin}, one gets
the following two possibility:

$(c)$ a subsequence $(w_{n_{k}})$ is bounded in $L^{\infty}_{\text{loc}}(\Omega\backslash S)$;

$(d)$ $(w_{n})$ converges uniformly to $-\infty$ on compact subset of $\Omega\backslash S$.

With the similar method in \cite{BrezisMerle}, we can exclude the possibility (c). It is important to note that we use the Green function $G(x)=\frac{1}{(Nk)^{\frac{1}{N-1}}}\log\frac{1}{F^{0}(x-x_{0})}$ of $Q_{N}$ in $\mathcal{W}_{R}(x_{0})$ instead of the ordinary Laplacian operator in a ball.

  Combing the above argument, the proof of Theorem 1.3 is completed.

\end{proof}

\section{Proof of Theorem 1.4}

In this section, we prove the Theorem 1.4 by analyzing the Pohozaev identity and using the expansion of Green function of $N$-anisotropic Laplacian operator,

\begin{proof}[Proof of Theorem 1.4]
Without loss of generality, we assume that $u_{n}$ blow up at $0$. Let $v_{n}$ be the weak solution of
\begin{eqnarray}
\begin{cases}
 -Q_{N}v_{n}=V_{n}e^{u_{n}}& \text{in}\ \Omega,\\
v_{n}=0&\ \text{on}\ \partial \Omega.
\end{cases}
\end{eqnarray}
and
$z_{n}=u_{n}-\min_{\partial \Omega}u_{n}-v_{n}$. One easily obtains that
\begin{eqnarray}
\begin{cases}
 -\widetilde{Q}_{N}z_{n}=0& \text{in}\ \Omega,\\
z_{n}=u_{n}-\min_{\partial \Omega}u_{n}&\ \text{on}\ \partial \Omega.
\end{cases}
\end{eqnarray}
By the standard quasilinear uniformly elliptic equation theory and condition (1.11), it follows that
\begin{equation}
||z_{n}||_{L^{\infty}(\Omega)}\leq ||z_{n}||_{L^{\infty}(\partial\Omega)}\leq C
 \ \ \text{and}\ \ ||\nabla z_{n}||_{L^{\infty}(\Omega)}\leq C.
\end{equation}
In addition, we may assume that $z_{n}\rightarrow z$ uniformly in $C^{0}(\overline{\Omega})\bigcap C_{\text{loc}}^{1}(\Omega)$. Let $Z_{n}=V_{n}\exp\{z_{n}+min_{\partial\Omega}u_{n}\}$, we get
\begin{equation}
-Q_{N}v_{n}=V_{n}e^{u_{n}}=Z_{n}e^{v_{n}}.
\end{equation}

  On the other hand, since 0 is the blow-up point, we obtain that  when $x_{n}\rightarrow 0$, $u_{n}(x_{n})=\max_{\overline{\Omega}}u_{n}\rightarrow\infty$.
Let $\delta_{n}=\exp\{-u_{n}(x_{n})/N\}$ and $\widetilde{u_{n}}(x)=u_{n}(\delta_{n}x+x_{n})+N\log \delta_{n}$.
Clearly, $\widetilde{u_{n}}\rightarrow \widetilde{u}$ locally in $C^{1}(\mathbb{R}^{N})$, where $\widetilde{u}$
is the solution of
\begin{eqnarray}
\begin{cases}
 -Q_{N}\widetilde{u}=V(0)e^{\widetilde{u}}& \text{in}\ \mathbb{R}^{N},\\
\widetilde{u}(0)=0,\\
\widetilde{u}\leq 0 &\text{in}\ \mathbb{R}^{N},\\
\int_{\mathbb{R}^{N}}e^{\widetilde{u}}<+\infty.
\end{cases}
\end{eqnarray}
If $V(0)=0$, then $\widetilde{u}$ must be a constant by Liouville type theorem, which contradicts $\int_{\mathbb{R}^{N}}e^{\widetilde{u}}<+\infty$. Thus it follows that $V(0)>0$, which deduces that $V_{n}$ has positive lower bound near the origin. Since $\nabla\log Z_{n}=\nabla\log V_{n}+\nabla z_{n}$,  by using (6.3) and the condition (1.10), one can get
\begin{equation}||\nabla \log Z_{n}||_{L^{\infty}(B_{r})}\leq C
\end{equation}
 for some $r>0$.

    From Theorem 1.3, we have for any $\psi\in C_{0}^{\infty}(\Omega)$,
    $$\int_{\Omega}-Q_{N}v_{n}\psi=\int_{\Omega}V_{n}e^{u_{n}}\psi\rightarrow \alpha \psi(0).$$

Let $\Omega_{1}=\{x:\ a|\nabla v_{n}(x)|\leq 1\}$, $\Omega_{2}=\{x:\ a|\nabla v_{n}(x)|>1\}$. By (2.2) and Proposition 2.1, we obtain

\begin{eqnarray}
\begin{split}
||\nabla v_{n}||_{L^{q}(\Omega)}&\leq||\nabla v_{n}||_{L^{q}(\Omega_{1})}+||\nabla v_{n}||_{L^{q}(\Omega_{2})}\\
&\leq C+\sup \{\int_{\Omega_{2}}\nabla v_{n}\nabla \psi:
||\psi||_{W_{0}^{1,q'}}=1\}\\
&\leq C+C\sup \{\int_{\Omega_{2}}F^{N-1}(\nabla v_{n})F_{\xi}(\nabla v_{n})\nabla \psi:
||\psi||_{W_{0}^{1,q'}}=1\}.\\
\end{split}
\end{eqnarray}

It is easy to see that $||\psi||_{L^{\infty}(\Omega)}\leq C$ by the Sobolev embedding.
Thus
$$\int_{\Omega_{2}}F^{N-1}(\nabla v_{n})F_{\xi}(\nabla v_{n})\nabla \psi=\int_{\Omega_{2}}-Q_{N}v_{n}\psi\leq ||V_{n}e^{u_{n}}||_{L^{1}(\Omega)}||\psi||_{L^{\infty}(\Omega)}\leq C.$$
Therefore $||\nabla v_{n}||_{L^{q}(\Omega)}\leq C$ for any $1<q<\frac{N}{N-1}$.

By Lemma 2.4 and 2.5, we have a unique Green function of

\begin{eqnarray}
\begin{cases}
 -Q_{N}G(\cdot,0)=\alpha \delta_{0}& \text{in}\ \Omega,\\
G(\cdot,0)=0 & \text{on}\ \partial \Omega,\\
\end{cases}
\end{eqnarray}
and $G$ has a decomposition
\begin{equation}G(x)=-\frac{\alpha}{(Nk)^{\frac{1}{N-1}}}\log F^{0}(x)+h(x),
\end{equation}
with
\begin{equation}
h(x)\in C^{0}(\Omega)\ \text{and}\ \lim_{x\rightarrow0}\ h(x)\ \text{exists},\ \lim_{x\rightarrow0}|x|\nabla h(x)=0.
\end{equation}
One easily obtains that $v_{n}\rightharpoonup G$ weakly in $W^{1,q}(\Omega)$.

By Corollary 4.4, from $\int_{\widetilde{\Omega}}V_{n}e^{u_{n}}\rightarrow 0$ for any $\widetilde{\Omega}\subset\subset\Omega\backslash {0}$, we obtain $||v_{n}^{+}||_{L^{\infty}(\widetilde{\Omega})}\leq C$. Thus $||v_{n}^{+}||_{C^{1,\beta}(\widetilde{\Omega})}\leq C$ for some $0<\beta<1$ by equation (6.4). Hence
$v_{n}\rightarrow G$ strongly in $C^{1,\beta}(\widetilde{\Omega})$.

Multiplying (6.4) by $<x,\nabla v_{n}>$ and integrating by parts, one can get the Pohozaev identity as follows:

\begin{eqnarray}
\begin{split}
&\int_{\partial \mathcal{W_{\epsilon}}}-F^{N-1}(\nabla v_{n})<F_{\xi}(\nabla v_{n},\nu)><x,\nabla v_{n}>\\
&\ \ \ \  +\frac{1}{N}F^{N}(\nabla v_{n})<x,\nu>\\
=&\int_{\partial \mathcal{W_{\epsilon}}}Z_{n}e^{v_{n}}<x,\nu>-
\int_{\mathcal{W_{\epsilon}}}NZ_{n}e^{v_{n}}+<x,\nabla\log Z_{n}>Z_{n}e^{v_{n}},
\end{split}
\end{eqnarray}
where $\nu=\frac{\nabla F^{0}}{|\nabla F^{0}|}$ is the unit outward normal.

Letting $n\rightarrow \infty$, the left-hand side of (6.11) converges to
\begin{eqnarray}
\begin{split}
I:&=\int_{\partial \mathcal{W_{\epsilon}}}-F^{N-1}(\nabla G)<F_{\xi}(\nabla G),\nu><x,\nabla G>\\
&\quad+\frac{1}{N}F^{N}(\nabla G)<x,\nu>.
\end{split}
\end{eqnarray}

By (6.9), (6.10) and Proposition 2.1, we obtain that on $\partial \mathcal{W_{\epsilon}}$,
\begin{eqnarray}
\begin{split}F^{N-1}(\nabla G)&=F^{N-1}\left(-\frac{\alpha}{(Nk)^{\frac{1}{N-1}}}\frac{\nabla F^{0}}{F^{0}}+o(\frac{1}{F^{0}})\right)\\
&\leq \left(\frac{\alpha}{\epsilon (Nk)^{\frac{1}{N-1}}}+o(\frac{1}{\epsilon})\right)^{N-1},
\end{split}
\end{eqnarray}

\begin{eqnarray}
\begin{split}
<F_{\xi}(\nabla G),\nu>&=<F_{\xi}(\nabla G),\frac{\nabla F^{0}}{|\nabla F^{0}|}>\\
&=\left<F_{\xi}(\nabla G),\left(-\frac{(Nk)^{\frac{1}{N-1}}}{\alpha}F^{0}\right)\frac{\nabla G-o(\frac{1}{F^{0}})}{|\nabla F^{0}|}\right>\\
&=-\frac{\epsilon(Nk)^{\frac{1}{N-1}}}{\alpha}\left(\frac{F(\nabla G)}{|\nabla F^{0}|}-\frac{o({\frac{1}{\epsilon}})}{|\nabla F^{0}|}\right)\\
&=-(1+o(1))\frac{1}{|\nabla F^{0}|},
\end{split}
\end{eqnarray}

\begin{eqnarray}
\begin{split}
<x,\nabla G>&=<x,-\frac{\alpha}{(Nk)^{\frac{1}{N-1}}}\frac{\nabla F^{0}}{F^{0}}+o(\frac{1}{F^{0}})>\\
&=-\frac{\alpha}{(Nk)^{\frac{1}{N-1}}}+<x,o(\frac{1}{\epsilon})>=-\frac{\alpha}{(Nk)^{\frac{1}{N-1}}}+o(1),\\
\end{split}
\end{eqnarray}
and
\begin{equation}
<x,\nu>=<x,\frac{\nabla F^{0}}{|\nabla F^{0}|}>=\frac{\epsilon}{|\nabla F^{0}|}.
\end{equation}

Substituting the above formula (6.13)-(6.16) into (6.12), one obtains
\begin{eqnarray}
\begin{split}
I=&\int_{\partial \mathcal{W_{\epsilon}}}\left(\frac{\alpha}{\epsilon (Nk)^{\frac{1}{N-1}}}+o(\frac{1}{\epsilon})\right)^{N-1}(1+o(1))\\
&\ \ \ \times\frac{1}{|\nabla F^{0}|}\left(-\frac{\alpha}{(Nk)^{\frac{1}{N-1}}}+o(1)\right)\\
&+\frac{1}{N}\left(\frac{\alpha}{\epsilon (Nk)^{\frac{1}{N-1}}}+o(\frac{1}{\epsilon})\right)^{N}\frac{\epsilon}{|\nabla F^{0}|}\\
=&-\frac{N-1}{N}\left\{\left(\frac{\alpha}{(Nk)^{\frac{1}{N-1}}}\right)^{N}+o(1)\right\}\frac{1}{\epsilon^{N-1}}
\int_{\partial \mathcal{W_{\epsilon}}}\frac{1}{|\nabla F^{0}|}\\
=&-(N-1)k\left(\frac{\alpha}{(Nk)^{\frac{1}{N-1}}}\right)^{N}+o(1).
\end{split}
\end{eqnarray}

Letting $n\rightarrow \infty$ and $\epsilon \rightarrow 0$, one deduce that the right-hand side of (6.11)
converges to $-N\alpha$ by (6.6) and $V_{n}e^{u_{n}}\rightharpoonup\alpha \delta_{0}$. Then we easily get $\alpha=\left(\frac{N^{N+1}k^{\frac{1}{N-1}}}{N-1}\right)^{\frac{1}{N-1}}$. Thus the proof of Theorem 1.4 is completed.
\end{proof}

\newpage


\end{document}